\documentclass[12pt,twoside]{amsart}
\usepackage{amssymb}
\usepackage{amscd}
\usepackage{amsmath}
\usepackage[all]{xy}

\title{On canonical bundle formulae and subadjunctions}
\author{Osamu Fujino and Yoshinori Gongyo} 
\date{2010/9/12, version 1.17}
\address{Department of Mathematics, Faculty of Science, 
Kyoto University, Kyoto 606-8502, Japan}
\email{fujino@math.kyoto-u.ac.jp}
\address{Graduate School of Mathematical Sciences, 
The University of Tokyo, 3-8-1 Komaba, Meguro, Tokyo, 153-8914 Japan.}
\email{gongyo@ms.u-tokyo.ac.jp}
\newcommand{\Bs}[0]{{\operatorname{Bs}}}
\newcommand{\Supp}[0]{{\operatorname{Supp}}}
\newtheorem{thm}{Theorem}[section]
\newtheorem{lem}[thm]{Lemma}
\newtheorem{cor}[thm]{Corollary}

\theoremstyle{definition}

\newtheorem{rem}[thm]{Remark}
\newtheorem*{ack}{Acknowledgments}       
\newtheorem*{notation}{Notation}         
\newtheorem{say}[thm]{}
\begin{document}
\bibliographystyle{amsalpha+}

\maketitle 

\begin{abstract} 
We consider a canonical bundle formula for 
generically finite proper surjective morphisms and 
obtain subadjunction formulae for minimal log canonical 
centers of log canonical pairs. 
We also treat related topics and applications. 
\end{abstract}

\tableofcontents

\section{Introduction}

The following lemma is one of the main results of this paper, 
which is missing in the literature. It is a canonical 
bundle formula for generically finite proper surjective morphisms. 

\begin{lem}
[Main Lemma]
\label{gongyo} 
Let $X$ and $Y$ be normal varieties and 
let $f:X\to Y$ be a generically finite proper surjective morphism. 
Let $\mathbb K$ be the rational number field $\mathbb Q$ or the real number field $\mathbb R$. 
Suppose that there exists an effective $\mathbb K$-divisor 
$\Delta$ on $X$  such that 
$(X, \Delta)$ is log canonical 
and that 
$K_X+\Delta \sim _{\mathbb K, f}0$. 
Then there exists an effective 
$\mathbb K$-divisor 
$\Gamma$ on $Y$ such that 
$(Y, \Gamma)$ is log canonical and 
that 
$$K_X+\Delta\sim _{\mathbb K}f^*(K_Y+\Gamma). 
$$
Moreover, if $(X, \Delta)$ is kawamata log terminal, 
then 
we can choose $\Gamma$ such that 
$(Y, \Gamma)$ is kawamata log terminal.  
\end{lem}

As an application of Lemma \ref{gongyo},  
we prove a subadjunction formula for minimal lc centers. 
It is a generalization of Kawamata's subadjunction formula (cf.~\cite[Theorem 1]{kawamata}). 
For a local version, see Theorem \ref{72} below. 

\begin{thm}[Subadjunction formula for minimal lc centers]\label{main-in} 
Let $\mathbb K$ be the rational number field $\mathbb Q$ or the 
real number field $\mathbb R$. 
Let $X$ be a normal projective variety and let $D$ be 
an effective $\mathbb K$-divisor on $X$  such that 
$(X, D)$ is log canonical. 
Let $W$ be a minimal log canonical center with respect to 
$(X, D)$. 
Then there exists an effective $\mathbb K$-divisor $D_W$ on $W$ such that 
$$
(K_X+D)|_W\sim _{\mathbb K} K_W+D_W
$$
and that the pair $(W, D_W)$ is kawamata log terminal. 
In particular, $W$ has only rational singularities. 
\end{thm}

We summarize the contents of this paper. 
Section \ref{sec2} is devoted to the proof of Lemma \ref{gongyo}. 
In Section \ref{sec6}, we discuss Ambro's canonical 
bundle formula for projective kawamata log terminal pairs with a generalization for 
$\mathbb R$-divisors (cf.~Theorem \ref{ambro-formula}). 
It is one of the key ingredients of 
the proof of Theorem \ref{main-in}. 
Although Theorem \ref{ambro-formula} is sufficient for 
applications in subsequent sections, we treat 
slight generalizations of Ambro's canonical bundle 
formula for 
projective log canonical pairs.  
In Section \ref{sec3}, we prove a 
subadjunction formula for 
minimal log canonical centers (cf.~Theorem \ref{main-in}), which is 
a generalization of Kawamata's subadjunction formula (cf.~\cite[Theorem 1]{kawamata}). 
In Section \ref{sec4}, we treat images of log Fano varieties by 
generically finite surjective morphisms as an application of Lemma \ref{gongyo}. 
Theorem \ref{41} is an answer to the question raised by Professor 
Karl Schwede (cf.~\cite[Remark 6.5]{ss}). 
In Section \ref{sec5}, 
we give a quick proof of the non-vanishing theorem for log canonical 
pairs as an application of Theorem \ref{main-in}, which is the main theorem of the first 
author's paper:~\cite{fujino2}. 
In Section \ref{sec7}, we prove a local version of our 
subadjunction formula for minimal log canonical centers (cf.~Theorem \ref{72}). 
It is useful for local studies of singularities of 
pairs. This local version does not directly follow from the global version:~Theorem \ref{main-in}. 
It is because we do not know how to compactify 
log canonical pairs. 

We close this introduction with the following notation. 
We also use the standard notation in \cite{kollar-mori}. 

\begin{notation} 
Let $\mathbb K$ be the real number field $\mathbb R$ or the rational 
number field $\mathbb Q$. 

Let $X$ be a normal variety and let $B$ be an effective $\mathbb K$-divisor 
such that $K_X+B$ is $\mathbb K$-Cartier. 
Then we can define the {\em{discrepancy}} $a(E, X, B)\in \mathbb K$ for every prime 
divisor $E$ {\em{over}} $X$. 
If $a(E, X, B) \geq -1$ (resp.~$>-1$) 
for every $E$, then $(X, B)$ is called {\em{log canonical}} 
(resp.~{\em{kawamata log terminal}}). 
We sometimes abbreviate log 
canonical (resp.~kawamata log terminal) to {\em{lc}} (resp.~{\em{klt}}). 

Assume that $(X, B)$ is log canonical. 
If $E$ is a prime divisor over $X$ 
such that $a(E, X, B)=-1$, then $c_X (E)$ is called a 
{\em{log canonical center}} ({\em{lc center}}, for short) of $(X, B)$, where $c_X(E)$ 
is the closure of the image of $E$ on $X$. 
For the basic properties of log canonical centers, 
see \cite[Theorem 2.4]{fujino2} or \cite[Section 9]{fujino3}. 

We note that $\sim _{\mathbb K}$ denotes {\em{$\mathbb K$-linear 
equivalence}} of $\mathbb K$-divisors. 
Let $f:X\to Y$ be a morphism between normal varieties and 
let $D$ be a $\mathbb K$-Cartier $\mathbb K$-divisor on $X$. 
Then $D$ is $\mathbb K$-linearly $f$-trivial, denoted by 
$D\sim _{\mathbb K, f}0$, if and only if there is 
a $\mathbb K$-Cartier $\mathbb K$-divisor $B$ on $Y$ such that 
$D\sim _{\mathbb K}f^*B$. 
 
The base locus of the linear system $\Lambda$ is denoted by 
$\Bs \Lambda$. 
\end{notation} 

\begin{ack}
The first author was partially supported by The Inamori 
Foundation and by the Grant-in-Aid for Young 
Scientists (A) $\sharp$20684001 from JSPS. 
He thanks Professor Karl Schwede for comments and questions.  
The second author was partially supported by the Research Fellowships of 
the Japan Society for the Promotion of Science for Young Scientists. 
\end{ack}

We will work over $\mathbb C$, the complex number field, throughout this paper. 

\section{Main lemma}\label{sec2}

In this section, we prove Lemma \ref{gongyo}. 

\begin{proof}[Proof of {\em{Lemma \ref{gongyo}}}]
Let 
$$
\begin{CD}
f:X@>{g}>>Z@>{h}>>Y 
\end{CD} 
$$ 
be the Stein factorization. 
By replacing $(X, \Delta)$ with 
$(Z, g_*\Delta)$, we can 
assume that $f:X\to Y$ is finite. 
Let $D$ be a $\mathbb{K}$-Cartier $\mathbb K$-divisor on 
$Y$ such that $K_X+\Delta \sim_{\mathbb{K}} f^*D$. 
We consider the following commutative diagram: 
$$
\xymatrix{
   X' \ar[r]^{\nu} \ar[d]_{f'} & X \ar[d]^{f} \\
   Y' \ar[r]_{\mu} & Y,
} 
$$
where 
\begin{itemize}
\item[(i)] $\mu$ is a resolution of singularities of $Y$,
\item[(ii)] there exists an open set $U \subseteq Y$ 
such that $\mu$ is isomorphic over $U$ and $f$ is \'etale over $U$. Moreover, 
$\mu^{-1}(Y-U)$ has a simple normal crossing 
support and $Y-U$ contains $\Supp f_*\Delta$, and 
\item[(iii)]$X'$ is the normalization of the irreducible 
component of $X \times_{Y} Y'$ which dominates 
$Y'$. 
In particular, $f'$ is finite.
\end{itemize}
Let $\Omega=\sum_i \delta_iD_i$ be a $\mathbb{K}$-divisor on $X'$ such that 
$$K_{X'}+\Omega=\nu^*(K_X+\Delta).$$ 
We consider the ramification formula:

$$K_{X'}=f'^*K_{Y'} + R,
$$
where $R=\sum_i (r_i-1)D_i$ is an effective $\mathbb{Z}$-divisor such that $r_i$ is the 
ramification index of $D_i$ for every $i$. 
Note that it suffices to show the above formula outside 
codimension two closed subsets of $X'$. 
Then it holds that 
$$(\mu \circ f')^*D \sim_{\mathbb{K}} f'^*K_{Y'} + R+\Omega.
$$
By pushing forward the above formula by $f'$, we see
$$\mathrm{deg}f' \cdot \mu^*D 
\sim_{\mathbb{K}} \mathrm{deg}f' \cdot K_{Y'} +f'_*(R+\Omega).
$$
We set 
$$\Gamma:=\frac{1}{\mathrm{deg}f'}\mu_*f'_*(R+\Omega)
$$
on $Y$.
Then  $\Gamma$ is effective since
$$\mu_*f'_*(R+\Omega)=f_*\nu_*(R+\Omega)=f_*(\nu_*R+ \Delta).
$$
Let $Y'\setminus \mu^{-1}U = \bigcup_{j}E_j$ be 
the irreducible decomposition, where $\sum_{j}E_j$ 
is a simple normal crossing divisor. We set 
$$I_j:=\{i| f'(D_i)=E_j \}.$$ 
The coefficient of $E_j$ in $f'_*(R+\Omega)$ is 
$$\frac{\sum_{i \in I_j}(r_i+\delta_i-1)\mathrm{deg}(f'|_{D_i})}{\mathrm{deg}f'}.
$$
Since $\delta_i\leq 1$, it holds that
$$\sum_{i \in I_j}(r_i+\delta_i-1)\mathrm{deg}(f'|_{D_i}) 
\leq \sum_{i \in I_j}r_i\mathrm{deg}(f'|_{D_i})=\mathrm{deg}f'.
$$
Thus $(Y,\Gamma)$ is log canonical since 
$K_{Y'}+f'_*(R+\Omega)=\mu^*(K_Y+\Gamma)$. 
Moreover, if $(X,\Delta)$ is kawamata log terminal, 
then $\delta_i < 1$. 
Hence $(Y,\Gamma)$ is kawamata log terminal.
\end{proof}

\section{Ambro's canonical bundle formula}\label{sec6}  
Theorem \ref{ambro-formula} is Ambro's canonical bundle formula for 
projective klt pairs (cf.~\cite[Theorem 4.1]{ambro}) 
with a generalization for $\mathbb R$-divisors. 
We need it for the proof of our subadjunction 
formula:~Theorem \ref{main-in}. 

\begin{thm}[Ambro's canonical bundle formula for 
projective klt pairs]\label{ambro-formula}Let $\mathbb K$ be 
the rational number field $\mathbb Q$ or 
the real number field $\mathbb R$. 
Let $(X, B)$ be a projective kawamata log terminal 
pair and let $f:X\to Y$ be a projective surjective 
morphism onto a normal projective variety 
$Y$ with connected fibers. 
Assume that 
$$
K_X+B\sim _{\mathbb K, f}0. 
$$
Then there exists an effective $\mathbb K$-divisor 
$B_Y$ on $Y$ such that 
$(Y, B_Y)$ is klt and 
$$K_X+B~\sim _{\mathbb K}f^*(K_Y+B_Y).$$  
\end{thm}
\begin{proof} 
If $\mathbb K=\mathbb Q$, then 
the statement is nothing but \cite[Theorem 4.1]{ambro}. 
From now on, we assume that 
$\mathbb K=\mathbb R$. 
Let $\sum _i B_i$ be the irreducible decomposition of $\Supp B$. 
We put $V=\underset{i}{\bigoplus}  \mathbb RB_i$. 
Then it is well known that 
$$
\mathcal L=\{ \Delta \in V\, | \, (X, \Delta)\ \text{is log canonical}\}
$$ 
is a rational polytope in $V$. 
We can also check that 
$$
\mathcal N=\{ \Delta \in \mathcal L\, |\, K_X+\Delta\ \text{is $f$-nef}\}
$$ 
is a rational polytope and $B\in \mathcal N$. 
We note that $\mathcal N$ is known as Shokurov's polytope. 
Therefore, we can write 
$$
K_X+B=\sum _{i=1}^k r_i (K_X+\Delta_i)
$$ 
such that 
\begin{itemize}
\item[(i)] $\Delta_i\in \mathcal N$ is an effective $\mathbb Q$-divisor on $X$ for 
every $i$, 
\item[(ii)] $(X, \Delta_i)$ is klt for every 
$i$, and 
\item[(iii)] $0<r_i<1$, $r_i\in \mathbb R$ for every $i$, and 
$\sum _{i=1}^k r_i =1$. 
\end{itemize}
Since $K_X+B$ is numerically $f$-trivial 
and $K_X+\Delta_i$ is $f$-nef for every $i$, 
$K_X+\Delta_i$ is numerically $f$-trivial for every $i$. 
Thus, $$\kappa (X_\eta, (K_{X}+\Delta_{i })_\eta)=\nu (X_\eta, (K_{X}+\Delta_{i})_\eta)=0$$ 
for every $i$, where $\eta$ is the generic point of $Y$, 
by Nakayama (cf.~\cite[Chapter V 
2.9.~Corollary]{nakayama}). See also \cite[Theorem 4.2]{ambro}.  
Therefore, $K_X+\Delta_i \sim _{\mathbb Q, f}0$ for every 
$i$ 
by \cite[Theorem 1.1]{fujino-kawamata}. 
By the case when $\mathbb K=\mathbb Q$, 
we can find an effective $\mathbb Q$-divisor 
$\Theta _i$ on $Y$ such that 
$(Y, \Theta _i)$ is klt and 
$$
K_X+\Delta_i\sim _{\mathbb Q}f^*(K_Y+\Theta_i) 
$$ 
for every $i$. 
By putting $B_Y=\sum _{i=1}^k r_i \Theta _i$, 
we obtain 
$$K_X+B\sim _{\mathbb R}f^*(K_Y+B_Y), $$ and 
$(Y, B_Y)$ is klt. 
\end{proof}

Corollary \ref{43} 
is a direct consequence of Theorem \ref{ambro-formula}. 

\begin{cor}\label{43} 
Let $\mathbb K$ be the rational number field $\mathbb Q$ or the 
real number field 
$\mathbb R$. 
Let $(X, B)$ be a log canonical 
pair and 
let $f:X\to Y$ be a projective 
surjective morphism between normal 
projective varieties. 
Assume that $$K_X+B\sim _{\mathbb K, f}0$$ 
and that every lc center of $(X, B)$ is dominant onto 
$Y$. 
Then we can find an effective 
$\mathbb K$-divisor 
$B_Y$ on $Y$ such that 
$(Y, B_Y)$ is 
kawamata log terminal 
and that 
$$K_X+B\sim _{\mathbb K} f^*(K_Y+B_Y).$$
\end{cor}

\begin{proof}
By taking a dlt blow-up (cf.~\cite[Theorem 10.4]{fujino3}), 
we can assume that $(X, B)$ is dlt. 
By replacing $(X, B)$ with its minimal 
lc center and taking the Stein 
factorization, we can 
assume that $(X, B)$ is klt and that $f$ has connected fibers 
(cf.~Lemma \ref{gongyo}). 
Therefore, we can take a desired $B_Y$ by Theorem \ref{ambro-formula}. 
\end{proof}

From now on, we treat Ambro's canonical bundle formula for 
projective log canonical pairs. We note that 
Theorem \ref{ambro-formula} is sufficient for 
applications in subsequent sections. 

\begin{say}[Observation]\label{44} 
Let $(X, B)$ be a log canonical pair and let $f:X\to Y$ be a projective 
surjective 
morphism between normal projective varieties with connected fibers. 
Assume that 
$K_X+B\sim _{\mathbb Q, f}0$ and 
that $(X, B)$ is kawamata log terminal over the generic 
point of $Y$. 
We can write 
$$
K_X+B\sim _{\mathbb Q} f^*(K_Y+M_Y+\Delta_Y)
$$
where $M_Y$ is the {\em{moduli $\mathbb Q$-divisor}} 
and $\Delta_Y$ is the {\em{discriminant $\mathbb Q$-divisor}}. 
For details, see, for example, \cite{ambro1}. 
It is conjectured that 
we can 
construct 
a commutative 
diagram 
$$
\xymatrix{
   X' \ar[r]^{\nu} \ar[d]_{f'} & X \ar[d]^{f} \\
   Y' \ar[r]_{\mu} & Y,
} 
$$
with the following 
properties. 
\begin{itemize}
\item[(i)] $\nu$ and $\mu$ 
are projective birational. 
\item[(ii)] $X'$ is normal and $K_{X'}+B_{X'}=\nu^*(K_X+B)$. 
\item[(iii)] $K_{X'}+B_{X'}\sim _{\mathbb Q}f'^*(K_{Y'}+M_{Y'}
+\Delta_{Y'})$ 
such that $Y'$ is smooth, the moduli $\mathbb Q$-divisor 
$M_{Y'}$ is semi-ample, and the discriminant 
$\mathbb Q$-divisor $\Delta_{Y'}$ has 
a simple normal crossing support. 
\end{itemize}
In the above properties, 
the non-trivial part is the semi-ampleness of 
$M_{Y'}$. We know that we can construct 
desired commutative diagrams 
of $f': X'\to Y'$ and $f:X\to Y$ when 
\begin{itemize}
\item[(1)] $\dim X-\dim Y=1$ (cf.~\cite[Theorem 5]{kawamata2} and so on), 
\item[(2)] $\dim Y=1$ (cf.~\cite[Theorem 0.1]{ambro1} and 
\cite[Theorem 3.3]{ambro}), 
\item[(3)] general fibers of $f$ are $K3$ surfaces, 
Abelian varieties, or 
smooth surfaces with $\kappa =0$ (cf.~\cite[Theorem 1.2, Theorem 6.3]{fujino-nagoya}), 
\end{itemize} 
and so on. 
We take a general member $D\in |mM_{Y'}|$ of the 
free linear system $|mM_{Y'}|$ where 
$m$ is a sufficiently large and 
divisible integer. 
We put 
$$
K_Y+B_Y=\mu_*(K_{Y'}+\frac{1}{m}D+\Delta_{Y'}). 
$$ 
Then it is easy to see that 
$$\mu^*(K_Y+B_Y)=K_{Y'}+\frac{1}{m}D+\Delta_{Y'},$$ 
$(Y, B_Y)$ is log canonical, and 
$$K_X+B\sim _{\mathbb Q}f^*(K_Y+B_Y).$$ 
By the above observation, we have Ambro's canonical bundle 
formula for projective log canonical pairs under some special assumptions. 

\begin{thm}
Let $(X, B)$ be a projective 
log canonical 
pair and let $f:X\to Y$ be a projective 
surjective morphism onto a normal 
projective variety $Y$ such that 
$K_X+B\sim _{\mathbb Q, f}0$. 
Assume that $\dim Y\leq 1$ or $\dim X-\dim Y\leq 1$. 
Then there exists an effective 
$\mathbb Q$-divisor 
$B_Y$ on $Y$ such that 
$(Y, B_Y)$ is log canonical 
and 
$$
K_X+B\sim _{\mathbb Q}f^*(K_Y+B_Y). 
$$
\end{thm}
\begin{proof}By taking 
a dlt blow-up (cf.~\cite[Theorem 10.4]{fujino3}), we can 
assume that $(X, B)$ is dlt. 
If necessary, by replacing $(X, B)$ with 
a suitable lc center of $(X, B)$ and by taking the Stein 
factorization 
(cf.~Lemma \ref{gongyo}), 
we can assume that $f:X\to Y$ has 
connected fibers and that $(X, B)$ is kawamata log terminal over 
the generic point of $Y$. 
We note that we can assume that 
$\dim Y=1$ or $\dim X-\dim Y=1$. 
By the arguments in 
\ref{44}, we can find an effective 
$\mathbb Q$-divisor $B_Y$ on $Y$ such that 
$(Y, B_Y)$ is log canonical and 
that $K_X+B\sim _{\mathbb Q}f^*(K_Y+B_Y)$. 
\end{proof} 
\end{say}

\section{Subadjunction for minimal log canonical centers}\label{sec3}

The following theorem is a generalization of Kawamata's subadjunction formula 
(cf.~\cite[Theorem 1]{kawamata}). 
Theorem \ref{main} is new even for threefolds. 
It is an answer to Kawamata's question (cf.~\cite[Question 1.8]{kawamata1}). 

\begin{thm}[Subadjunction formula for minimal lc centers]\label{main} 
Let $\mathbb K$ be the rational number field $\mathbb Q$ or the 
real number field $\mathbb R$. 
Let $X$ be a normal projective variety and let $D$ be 
an effective $\mathbb K$-divisor on $X$  such that 
$(X, D)$ is log canonical. 
Let $W$ be a minimal log canonical center with respect to 
$(X, D)$. 
Then there exists an effective $\mathbb K$-divisor $D_W$ on $W$ such that 
$$
(K_X+D)|_W\sim _{\mathbb K} K_W+D_W
$$
and that the pair $(W, D_W)$ is kawamata log terminal. 
In particular, $W$ has only rational singularities. 
\end{thm}

\begin{rem}\label{32} 
In \cite[Theorem 1]{kawamata}, 
Kawamata proved 
$$
(K_X+D+\varepsilon H)|_W\sim _{\mathbb Q}K_W+D_W, 
$$ 
where $H$ is an ample Cartier divisor on $X$ and 
$\varepsilon$ is a positive rational number, 
under the extra assumption that 
$D$ is a $\mathbb Q$-divisor and there exists an effective $\mathbb Q$-divisor 
$D^o$ such that $D^o<D$ and that $(X, D^o)$ is kawamata log terminal. 
Therefore, Kawamata's theorem claims nothing when $D=0$. 
\end{rem}

\begin{proof}[Proof of {\em{Theorem \ref{main}}}]
By taking a dlt blow-up (cf.~\cite[Theorem 10.4]{fujino3}), 
we can take a projective birational morphism 
$f:Y\to X$ from a normal projective variety $Y$ with the 
following properties. 
\begin{itemize}
\item[(i)] $K_Y+D_Y=f^*(K_X+D)$. 
\item[(ii)] $(Y, D_Y)$ is a $\mathbb Q$-factorial 
dlt pair. 
\end{itemize}
We can take a minimal lc center $Z$ of $(Y, D_Y)$ such that 
$f(Z)=W$. 
We note that 
$K_Z+D_Z=(K_Y+D_Y)|_Z$ is klt since 
$Z$ is a minimal lc center of the dlt pair 
$(Y, D_Y)$. 
Let 
$$
\begin{CD}
f:Z@>{g}>> V@>{h}>> W
\end{CD}
$$ 
be the Stein factorization of $f: Z\to W$. 
By the construction, we can write 
$$
K_Z+D_Z\sim _{\mathbb K}f^*A 
$$
where $A$ is a $\mathbb K$-divisor 
on $W$ such that 
$A\sim _{\mathbb K} (K_X+D)|_W$. 
We note that $W$ is normal (cf.~\cite[Theorem 2.4 (4)]{fujino2}).  
Since $(Z, D_Z$) is klt, 
we can take an effective $\mathbb K$-divisor 
$D_V$ on $V$ such that 
$$
K_V+D_V\sim _{\mathbb K}h^*A$$ 
and that $(V, D_V)$ is klt by Theorem \ref{ambro-formula}. 
By Lemma \ref{gongyo}, 
we can find an effective $\mathbb K$-divisor 
$D_W$ on $W$ such that 
$$
K_W+D_W\sim _{\mathbb K} A\sim_{\mathbb K} 
(K_X+D)|_W 
$$ and 
that $(W, D_W)$ is klt. 
\end{proof}

\section{Log Fano varieties}\label{sec4}

In this section, we give an easy application of Lemma \ref{gongyo}. 
Theorem \ref{41} is an answer to the question raised by 
Karl Schwede 
(cf.~\cite[Remark 6.5]{ss}). For related topics, see \cite[Section 3]{fujino-gongyo}. 

\begin{thm}\label{41}
Let $(X, \Delta)$ be a projective 
klt pair such that 
$-(K_X+\Delta)$ is ample. 
Let $f:X\to Y$ be a generically finite 
surjective morphism to a normal projective 
variety $Y$. 
Then we can find an effective $\mathbb Q$-divisor 
$\Delta_Y$ on $Y$ such that 
$(Y, \Delta_Y)$ is klt and 
$-(K_Y+\Delta_Y)$ is ample. 
\end{thm}
\begin{proof}
Without loss of generality, we can assume that 
$\Delta$ is a $\mathbb Q$-divisor by perturbing 
the coefficients of $\Delta$. 
Let $H$ be a general very ample Cartier divisor 
on $Y$ and let $\varepsilon$ be a sufficiently 
small positive 
rational number. 
Then $K_X+\Delta+\varepsilon f^*H$ is anti-ample and $(X, \Delta+\varepsilon f^*H)$ is klt. 
We can take an effective 
$\mathbb Q$-divisor 
$\Theta$ on $X$ such that 
$m\Theta$ is a general 
member of the free linear system $|-m(K_X+\Delta+\varepsilon f^*H)|$ where 
$m$ is a sufficiently large and divisible integer. 
Then $$K_X+\Delta+\varepsilon f^*H+\Theta\sim _{\mathbb Q}0.$$  
Let $\delta$ be a positive rational number 
such that 
$0<\delta<\varepsilon$. 
Then $$K_X+\Delta+(\varepsilon -\delta)f^*H+\Theta\sim _{\mathbb Q}
f^*(-\delta H).$$  
By Lemma \ref{gongyo}, 
we can find an effective $\mathbb Q$-divisor 
$\Delta_Y$ on $Y$ such that 
$$K_Y+\Delta_Y\sim _{\mathbb Q}-\delta H$$ and 
that $(Y, \Delta_Y)$ is klt. 
We note that $$
-(K_Y+\Delta_Y)\sim _{\mathbb Q} \delta H
$$ 
is ample. 
\end{proof}

By combining Theorem \ref{41} with \cite[Theorem 3.1]{fujino-gongyo}, we can 
easily obtain the following corollary. 

\begin{cor}
Let $(X, \Delta)$ be a projective klt pair 
such that 
$-(K_X+\Delta)$ is ample. 
Let $f:X\to Y$ be a projective 
surjective morphism 
onto a normal projective variety $Y$. 
Then we can find an effective $\mathbb Q$-divisor 
$\Delta_Y$ on $Y$ such that 
$(Y, \Delta_Y)$ is klt and $-(K_Y+\Delta_Y)$ is ample.  
\end{cor}

\section{Non-vanishing theorem for log canonical pairs}\label{sec5}
The following theorem is the main result of \cite{fujino2}. 
It is almost equivalent to the base point free theorem for 
log canonical pairs. 
For details, see \cite{fujino2}.  

\begin{thm}[Non-vanishing theorem]\label{non}  
Let $X$ be a normal projective 
variety and let $B$ be an effective $\mathbb Q$-divisor on $X$ such that $(X, B)$ 
is log canonical.
Let $L$ be a nef Cartier divisor on $X$. Assume that $aL-(K_X+B)$ 
is ample for some $a>0$. 
Then the base locus of the linear system $|mL|$ contains no lc 
centers of $(X, B)$ for every $m\gg 0$, 
that is, there is a positive integer $m_0$ such
that $\Bs |mL|$contains no lc centers of $(X, B)$ for every $m\geq m_0$.
\end{thm}

Here, we give a quick proof of Theorem \ref{non} by using Theorem \ref{main}. 

\begin{proof}
Let $W$ be any minimal lc center of the pair $(X, B)$. 
It is sufficient to prove that 
$W$ is not contained in $\Bs|mL|$ for $m\gg 0$. 
By Theorem \ref{main}, we can find an effective $\mathbb Q$-divisor 
$B_W$ on $W$ such that 
$(W, B_W)$ is klt and $K_W+D_W\sim _{\mathbb Q}(K_X+B)|_W$. 
Therefore, 
$aL|_W-(K_W+B_W)\sim _{\mathbb Q}(aL-(K_X+B))|_W$ 
is ample. 
By the Kawamata--Shokurov base point free theorem, 
$|mL|_W|$ is free for $m \gg 0$. 
By \cite[Theorem 2.2]{fujino2}, 
$$
H^0(X, \mathcal O_X(mL))\to H^0(W, \mathcal O_W(mL))
$$ 
is surjective 
for $m \geq a$. 
Therefore, $W$ is not contained in $\Bs|mL|$ for 
$m \gg 0$. 
\end{proof}

\begin{rem}
The above proof of Theorem \ref{non} 
is shorter than the original proof in \cite{fujino2}. 
However, the proof of Theorem \ref{main} depends on 
very deep results such 
as existence of 
dlt blow-ups. 
The proof of Theorem \ref{non} 
in \cite{fujino2} only depends on various well-prepared vanishing theorems and 
standard techniques. 
\end{rem}

\section{Subadjunction formula:~local version}\label{sec7} 

In this section, we give a local version of 
our subadjunction formula for minimal log canonical 
centers. Theorem \ref{local} is a local version of 
Ambro's canonical bundle formula for kawamata log terminal 
pairs:~Theorem \ref{ambro-formula}.  
It is essentially \cite[Theorem 1.2]{fujino1}. 

\begin{thm}\label{local} 
Let $\mathbb K$ be 
the rational number field $\mathbb Q$ or 
the real number field $\mathbb R$. 
Let $(X, B)$ be a kawamata log terminal 
pair and let $f:X\to Y$ be a proper surjective 
morphism onto a normal {\em{affine}} variety 
$Y$ with connected fibers. 
Assume that 
$$
K_X+B\sim _{\mathbb K, f}0. 
$$
Then there exists an effective $\mathbb K$-divisor 
$B_Y$ on $Y$ such that 
$(Y, B_Y)$ is klt and 
$$K_X+B~\sim _{\mathbb K}f^*(K_Y+B_Y).$$  
\end{thm}

\begin{proof}[Sketch of the proof]
First, we assume that $\mathbb K=\mathbb Q$. 
In this case, the proof of \cite[Theorem 1.2]{fujino1} 
works with some minor modifications. 
We note that $M$ in the proof of \cite[Theorem 1.2]{fujino1} 
is $\mu$-nef. We also note that 
we can assume $H=0$ in \cite[Theorem 1.2]{fujino1} since 
$Y$ is affine. 
Next, we assume that $\mathbb K=\mathbb R$. 
In this case, the reduction argument in the proof of Theorem \ref{ambro-formula} 
can be applied. 
So, we obtain the desired formula. 
\end{proof}

By Theorem \ref{local}, we can obtain a local version of 
Theorem \ref{main}. 
The proof of Theorem \ref{main} works without any modifications. 

\begin{thm}[Subadjunction formula for minimal lc centers:~local version]\label{72} 
Let $\mathbb K$ be the rational number field $\mathbb Q$ or the 
real number field $\mathbb R$. 
Let $X$ be a normal {\em{affine}} variety and let $D$ be 
an effective $\mathbb K$-divisor on $X$  such that 
$(X, D)$ is log canonical. 
Let $W$ be a minimal log canonical center with respect to 
$(X, D)$. 
Then there exists an effective $\mathbb K$-divisor $D_W$ on $W$ such that 
$$
(K_X+D)|_W\sim _{\mathbb K} K_W+D_W
$$
and that the pair $(W, D_W)$ is kawamata log terminal. 
In particular, $W$ has only rational singularities. 
\end{thm}

Theorem \ref{72} does not directly follow from Theorem \ref{main}. 
It is because we do not know how to compactify log canonical pairs. 



\begin{thebibliography}{KM}

\bibitem[A1]{ambro1} 
F.~Ambro, 
Shokurov's boundary property, 
J. Differential Geom. {\textbf{67}} (2004), no. 2, 229--255.

\bibitem[A2]{ambro} F.~Ambro, 
The moduli $b$-divisor of an lc-trivial fibration, 
Compos. Math. {\textbf{141}} (2005), no. 2, 385--403.

\bibitem[F1]{fujino1} 
O.~Fujino, Applications of 
Kawamata's positivity theorem, 
Proc. Japan Acad. Ser. A Math. Sci. {\textbf{75}} (1999), no. 6, 75--79.

\bibitem[F2]{fujino-nagoya} 
O.~Fujino, 
A canonical bundle formula for certain algebraic fiber spaces and its applications, 
Nagoya Math. J. {\textbf{172}} (2003), 129--171.

\bibitem[F3]{fujino-kawamata} 
O.~Fujino, On Kawamata's theorem, 
to appear in the proceeding of the \lq\lq Classification of Algebraic Varieties\rq\rq 
conference, Schiermonnikoog, Netherlands, May 10--15, 2009.

\bibitem[F4]{fujino2} 
O.~Fujino, Non-vanishing theorem for 
log canonical pairs, to appear in Journal of Algebraic Geometry. 

\bibitem[F5]{fujino3} 
O.~Fujino, Fundamental theorems for the log minimal model program, 
to appear in Publ. Res. Inst. Math. Sci. 

\bibitem[FG]{fujino-gongyo} 
O.~Fujino, Y.~Gongyo, 
On images of weak Fano manifolds, 
preprint 2010. 

\bibitem[K1]{kawamata1} 
Y.~Kawamata, 
On Fujita's freeness conjecture for $3$-folds and $4$-folds, 
Math. Ann. {\textbf{308}} (1997), no. 3, 491--505.

\bibitem[K2]{kawamata2} 
Y.~Kawamata, 
Subadjunction of log canonical divisors 
for a subvariety of codimension $2$, 
{\em{Birational algebraic geometry}} (Baltimore, MD, 1996), 
79--88, Contemp. Math., {\textbf{207}}, Amer. Math. Soc., Providence, RI, 1997.

\bibitem[K3]{kawamata} 
Y.~Kawamata, 
Subadjunction of log canonical divisors, 
II. Amer. J. Math. {\textbf{120}} (1998), no. 5, 893--899. 

\bibitem[KM]{kollar-mori} 
J.~Koll\'ar, S. Mori, {\em{Birational geometry of algebraic varieties}}, Cambridge
Tracts in Mathematics, Vol. {\textbf{134}}, 1998.

\bibitem[N]{nakayama} 
N.~Nakayama, 
{\em{Zariski-decomposition and abundance}}, 
MSJ Memoirs, {\textbf{14}}. Mathematical Society of Japan, Tokyo, 2004. 

\bibitem[SS]{ss} 
K.~Schwede, K.~E.~Smith, 
Globally $F$-regular and log Fano varieties, 
Adv. Math. {\textbf{224}} (2010), no. 3, 863--894.

\end{thebibliography}
\end{document}